\documentclass[a4paper,12pt]{article}
\usepackage[T1]{fontenc}
\usepackage{kpfonts, comment, amsfonts, amsmath, amssymb, amsthm, mathtools, geometry, setspace, enumitem}

\usepackage[svgnames,dvipsnames]{xcolor}
\definecolor{theme}{RGB}{170, 0, 0}

\usepackage[unicode=true]{hyperref}
\hypersetup{% Links and pdf setup
colorlinks=true, 
linkcolor=DarkGreen, % Colour of internal links
filecolor=black,
citecolor=DarkGreen, % Colour of internal citations
urlcolor=DarkBlue,
pdftitle={Treewidth of Products of Graphs with High Treewidth}, % Title in pdf view
pdfauthor={Kaul}
}

\usepackage[nameinlink,capitalise]{cleveref}

\usepackage[all]{hypcap}

\usepackage[longnamesfirst,numbers,sort&compress]{natbib}
\makeatletter
\def\NAT@spacechar{~}
\makeatother

\usepackage{tikz}
\usetikzlibrary{perspective}

\allowdisplaybreaks
\setlist{nolistsep}

\DeclareMathOperator{\tw}{tw}
\DeclareMathOperator{\bn}{bn}
\DeclareMathOperator{\sbn}{sbn}
\DeclareMathOperator{\pw}{pw}
\DeclareMathOperator{\ord}{ord}
\DeclareMathOperator{\had}{had}
\DeclareMathOperator{\degen}{degen}
\DeclareMathOperator{\cprod}{\kern0.1pt \square \kern0.1pt}

\renewcommand{\leq}{\leqslant}
\renewcommand{\geq}{\geqslant}
\renewcommand{\emptyset}{\varnothing}

\DeclarePairedDelimiter\floor{\lfloor}{\rfloor}
\DeclarePairedDelimiter\set{\{}{\}}

\let\angle\relax
\DeclarePairedDelimiter\angle{\langle}{\rangle}

\DeclareMathOperator{\lex}{\vcenter{\hbox{$\circ$}}}

\usepackage[mathscr]{euscript}
\DeclareMathOperator{\sB}{\mathscr{B}}
\DeclareMathOperator{\sS}{\mathscr{S}}

\newtheorem{theorem}{Theorem}
\newtheorem{lemma}[theorem]{Lemma}
\newtheorem{claim}{Claim}
\newtheorem{fact}[theorem]{Fact}
\crefname{fact}{Fact}{Facts}

\newtheorem{corollary}[theorem]{Corollary}
\newtheorem{question}[theorem]{Question}

\newenvironment{subproof}[1][Subproof]{
  
  \begin{proof}[#1]
}{
  \end{proof}
}

\newcommand{\define}[1]{{\color{theme}\textit{#1}}} 
\newcommand{\mathdefine}[1]{\mathcolor{theme}{#1}}

\newcommand\blfootnote[1]{%
  \begingroup
  \renewcommand\thefootnote{}\footnote{#1}%
  \addtocounter{footnote}{-1}%
  \endgroup
}

\title{\bf \LARGE Treewidth of Products of Graphs\\
with High Treewidth}

\author{\vspace{0.15cm} Raj Kaul}
\date{\vspace{-1.7cm}}

\begin{document}
\maketitle

\blfootnote{\textit{Date:} \today.\\School of Mathematics, Monash University, Melbourne, Australia (\texttt{raj.kaul@monash.edu}). Supported by an Australian Government Research Training Program Scholarship.}

\begin{abstract}
\noindent
Treewidth is the standard measure for how ``tree-like'' a graph is.
This paper studies how the treewidth of a product graph depends on the treewidth of its factors.
Kozawa, Otachi, and Yamazaki [2014] and Hickingbotham and Wood [2025] independently showed that $\tw(G\boxtimes H)\geq (\tw(G)+1)\had(H)-1$ for all graphs $G$ and $H$, where $\had(H)$ is the Hadwiger number of $H$. We improve this bound to $\tw(G\boxtimes H)\geq (\tw(G)+1)(\tw(H)+1)-1$, thereby solving an open problem of Hickingbotham and Wood.
We also prove analogous product inequalities for pathwidth, Cartesian products, and strict bramble number, which is a parameter that is tied to treewidth.
As an application of our results, we show that products of expanders have large subgraphs that are expanders.
\end{abstract}

\section{Introduction}
\label{sec:intro}

Treewidth\footnote{We consider simple, finite, and undirected graphs. A set $A\subseteq V(G)$ is \define{connected} in $G$ if $G[A]$ is connected. If $T$ is a tree and $x,y\in V(T)$, then \define{$xTy$} denotes the unique path in $T$ whose set of endpoints is $\set{x,y}$. 
A \define{tree-decomposition} of a graph $G$ is a pair $(T,\beta)$ consisting of a tree $T$ and a function $\beta:V(T)\rightarrow 2^{V(G)}$ with the following two properties: (\define{vertex-property}) for every vertex $v\in V(G)$, the set $\set{t\in V(T) : v\in \beta(t)}$ is non-empty and connected in $T$; and (\define{edge-property}) for every edge $uv\in E(G)$, there exists $t\in V(T)$ such that $\set{u,v}\subseteq \beta(t)$. The vertices of $T$ are called \define{nodes}, and the sets $\beta(t)$ are called \define{bags}. If the underlying tree $T$ is a path, then $(T,\beta)$ may be called a \define{path-decomposition}. The \define{width} of $(T,\beta)$ is $\max_{t\in V(T)} |\beta(t)|-1$, and the \define{treewidth} of $G$, denoted by \define{$\tw(G)$}, is the minimum width of a tree-decomposition of $G$. The pathwidth of $G$, denoted by \define{$\pw(G)$}, is the minimum width of a path-decomposition of $G$.}
is a foundational parameter in structural and algorithmic graph theory \cite{BODLAENDER19981, ReedAlgAspects, TreewidthParams2017} that measures the structural complexity of a graph.
Informally, graphs with small treewidth are more ``tree-like''. 
Grids are the canonical graphs with large treewidth, since the $n\times n$-grid has treewidth $n$ and, by the Grid-Minor Theorem of \citet{ROBERTSON198692}, every graph with large enough treewidth contains the $n\times n$-grid as a minor.\footnote{A graph $H$ is a \define{minor} of a graph $G$ if $H$ can be obtained from $G$ by a sequence of vertex deletions, edge deletions, and edge contractions.}
Inspired by the fact that grids are Cartesian products of paths, this paper studies the treewidth of various products of general graphs.

Our first contribution, proven in \cref{sec:treewidth-of-product-graphs}, is a lower bound on the treewidth of a strong product. The \define{strong product} of graphs $G$ and $H$, denoted by \define{$G\boxtimes H$}, is the graph with vertex set $V(G)\times V(H)$, where $(g_1, h_1)(g_2, h_2)\in E(G\boxtimes H)$ if and only if $g_1=g_2$ and $h_1h_2\in E(H)$; or $g_1g_2\in E(G)$ and $h_1=h_2$; or $g_1g_2\in E(G)$ and $h_1h_2\in E(H)$.

\begin{theorem}
\label{thm:main-strong-product}
For all graphs $G$ and $H$, $\tw(G\boxtimes H) \geq (\tw(G)+1)(\tw(H)+1)-1$.
\end{theorem}

\cref{thm:main-strong-product} improves upon the bound $\tw(G\boxtimes H)\geq (\tw(G)+1)\had(H)-1$, where \define{$\had(H)$} denotes the Hadwiger number of $H$,\footnote{The \define{Hadwiger number} of a graph $H$ is the maximum integer $t$ such that $K_t$ is a minor of $H$. Note that since treewidth is minor-monotone, $\tw(H) \geq \tw(K_{\had(H)}) = \had(H)-1$.}
which was shown independently by \citet{KOZAWA2014251} and \citet{HickingbothamStructuralPropertiesofGraphProducts}. Furthermore, \cref{thm:main-strong-product} solves an open problem of \citet{HickingbothamStructuralPropertiesofGraphProducts}, who asked whether there exists a constant $c>0$ such that $\tw(G\boxtimes H) \geq c\tw(G)\tw(H)$ for all graphs $G$ and $H$. \cref{thm:main-strong-product} says the answer is `yes' with $c=1$.

Our second contribution, proven in \cref{sec:pathwidth-of-strong-product}, is that \cref{thm:main-strong-product} holds when treewidth is replaced by pathwidth:

\begin{theorem}
\label{thm:main-strong-product-pw}
For all graphs $G$ and $H$, $\pw(G\boxtimes H) \geq (\pw(G)+1)(\pw(H)+1)-1$.
\end{theorem}

\cref{thm:main-strong-product,thm:main-strong-product-pw} 
are optimal since $\tw(K_n) = \pw(K_n)=n-1$
and $\tw(K_m \boxtimes K_n) = \tw(K_{mn}) = mn-1 = (\tw(K_m)+1)(\tw(K_n)+1)-1$.

In addition to the strong product, this paper also studies
Cartesian products and lexicographic products, which we now introduce. The \define{Cartesian product} of graphs $G$ and $H$, denoted by \define{$G\cprod H$}, is the graph with vertex set $V(G)\times V(H)$, where $(g_1, h_1)(g_2, h_2)\in E(G\cprod H)$ if and only if $g_1=g_2$ and $h_1 h_2\in E(H)$; or $g_1 g_2\in E(G)$ and $h_1=h_2$. The \define{lexicographic product} of graphs $G$ with $H$, denoted by \define{$G\lex H$}, is the graph with vertex set $V(G)\times V(H)$, where $(g_1, h_1)(g_2, h_2) \in E(G\lex H)$ if and only if $g_1g_2\in E(G)$; or $g_1=g_2$ and $h_1h_2\in E(H)$.

Both $\cprod$ and $\boxtimes$ are commutative (through isomorphism). On the other hand, $\lex$ is not commutative. As illustrated in \Cref{fig:product-examples}, the basic relationship between the different products is $G\cprod H\subseteq G\boxtimes H \subseteq G\lex H$.

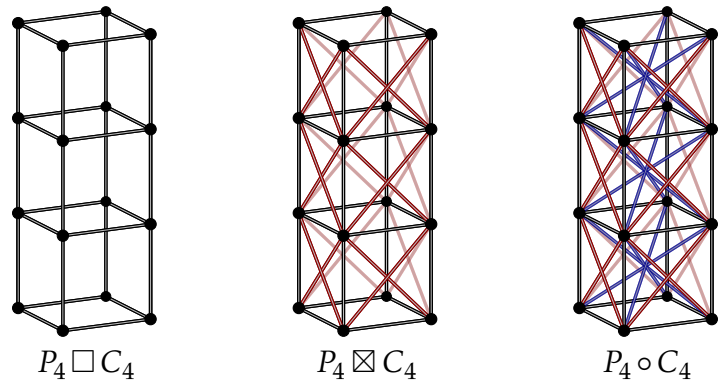
\begin{figure}[h]
\centering

\def\s{1.3} % tikz scale

\def\gcol{gray!50} % double colour

\begin{tikzpicture}[vertex/.style={circle,draw,fill=black,inner sep=1.5pt, black}, /tikz/xscale=\s, /tikz/yscale=\s, 3d view={-27}{15}, bline/.style={line width=0.5pt, double=\gcol, double distance=0.25pt}]
%%
% Cartesian
%%

\foreach \i in {0,1}{
\foreach \j in {0,1}{
\foreach \x in {0,1,2,3}{
\foreach \l in {0,1}{
\foreach \k in {0,1}{
\foreach \y in {0,1,2,3}{

    \ifnum \inteval{\i+\j}>0
    \ifnum \inteval{\k+\l}>0
    
    \ifnum \x=\y
        \ifnum \inteval{(\i-\l)*(\i-\l)+(\j-\k)*(\j-\k)}=1
            \draw[bline] (\i,\j,\x) to (\l,\k,\y);
        \fi
    \fi

    \ifnum \inteval{(\x-\y)*(\x-\y)}=1
        \ifnum \i=\l
            \ifnum \j=\k
                \draw[bline] (\i,\j,\x) to (\l,\k,\y);
            \fi
        \fi
    \fi

    \fi
    \fi
}}}
}}}

\foreach \i in {0,1}{
\foreach \j in {0,1}{
\foreach \x in {0,1,2,3}{
    \ifnum \inteval{\i+\j}>0
        \node[vertex, inner sep=1.3pt] at (\i,\j,\x) {};
    \fi
}}}

% front

\foreach \i in {0,1}{
\foreach \j in {0,1}{
\foreach \x in {0,1,2,3}{
\foreach \l in {0,1}{
\foreach \k in {0,1}{
\foreach \y in {0,1,2,3}{

    \ifnum \inteval{\i+\j}<2
    \ifnum \inteval{\k+\l}<2
    
    \ifnum \x=\y
        \ifnum \inteval{(\i-\l)*(\i-\l)+(\j-\k)*(\j-\k)}=1
            \draw[bline] (\i,\j,\x) to (\l,\k,\y);
        \fi
    \fi

    \ifnum \inteval{(\x-\y)*(\x-\y)}=1
        \ifnum \i=\l
            \ifnum \j=\k
                \draw[bline] (\i,\j,\x) to (\l,\k,\y);
            \fi
        \fi
    \fi

    \fi
    \fi
}}}
}}}

\foreach \i in {0,1}{
\foreach \j in {0,1}{
\foreach \x in {0,1,2,3}{
    \ifnum \inteval{\i+\j}<2
        \node[vertex] at (\i,\j,\x) {};
    \fi
}}}
\end{tikzpicture}
\qquad\qquad %
\begin{tikzpicture}[vertex/.style={circle,draw,fill=black,inner sep=1.5pt, black}, /tikz/xscale=\s, /tikz/yscale=\s, 3d view={-27}{15}, bline/.style={line width=0.5pt, double=\gcol, double distance=0.25pt}, line/.style={line width=0.5pt, DarkRed, double=\gcol, double distance=0.25pt}]
%%
% Strong
%%

% back
\foreach \i in {0,1}{
\foreach \j in {0,1}{
\foreach \x in {0,1,2,3}{
\foreach \l in {0,1}{
\foreach \k in {0,1}{
\foreach \y in {0,1,2,3}{

    \ifnum \inteval{\i+\j}>0
    \ifnum \inteval{\k+\l}>0
    
    \ifnum \x=\y
        \ifnum \inteval{(\i-\l)*(\i-\l)+(\j-\k)*(\j-\k)}=1
            \draw[bline] (\i,\j,\x) to (\l,\k,\y);
        \fi
    \fi

    \ifnum \inteval{(\x-\y)*(\x-\y)}=1
        \ifnum \i=\l
            \ifnum \j=\k
                \draw[bline] (\i,\j,\x) to (\l,\k,\y);
            \fi
        \fi
    \fi

    \ifnum \inteval{(\x-\y)*(\x-\y)}=1
        \ifnum \inteval{(\i-\l)*(\i-\l)+(\j-\k)*(\j-\k)}=1
            \draw[line, opacity=0.2] (\i,\j,\x) to (\l,\k,\y);
        \fi
    \fi

    \fi
    \fi
}}}
}}}

\foreach \i in {0,1}{
\foreach \j in {0,1}{
\foreach \x in {0,1,2,3}{
    \ifnum \inteval{\i+\j}>0
        \node[vertex, inner sep=1.3pt] at (\i,\j,\x) {};
    \fi
}}}

% front

\foreach \i in {0,1}{
\foreach \j in {0,1}{
\foreach \x in {0,1,2,3}{
\foreach \l in {0,1}{
\foreach \k in {0,1}{
\foreach \y in {0,1,2,3}{

    \ifnum \inteval{\i+\j}<2
    \ifnum \inteval{\k+\l}<2
    
    \ifnum \x=\y
        \ifnum \inteval{(\i-\l)*(\i-\l)+(\j-\k)*(\j-\k)}=1
            \draw[bline] (\i,\j,\x) to (\l,\k,\y);
        \fi
    \fi

    \ifnum \inteval{(\x-\y)*(\x-\y)}=1
        \ifnum \i=\l
            \ifnum \j=\k
                \draw[bline] (\i,\j,\x) to (\l,\k,\y);
            \fi
        \fi
    \fi

    \ifnum \inteval{(\x-\y)*(\x-\y)}=1
        \ifnum \inteval{(\i-\l)*(\i-\l)+(\j-\k)*(\j-\k)}=1
            \draw[line] (\i,\j,\x) to (\l,\k,\y);
        \fi
    \fi

    \fi
    \fi
}}}
}}}

\foreach \i in {0,1}{
\foreach \j in {0,1}{
\foreach \x in {0,1,2,3}{
    \ifnum \inteval{\i+\j}<2
        \node[vertex] at (\i,\j,\x) {};
    \fi
}}}
\end{tikzpicture}
\qquad\qquad %
\begin{tikzpicture}[vertex/.style={circle,draw,fill=black,inner sep=1.5pt, black}, /tikz/xscale=\s, /tikz/yscale=\s, 3d view={-27}{15}, bline/.style={line width=0.5pt, double=\gcol, double distance=0.25pt}, line/.style={line width=0.5pt, DarkRed, double=\gcol, double distance=0.25pt}, lline/.style={line width=0.5pt, DarkBlue, double=\gcol, double distance=0.25pt}]
%%
% Lexicographic
%%

% back (strong)
\foreach \i in {0,1}{
\foreach \j in {0,1}{
\foreach \x in {0,1,2,3}{
\foreach \l in {0,1}{
\foreach \k in {0,1}{
\foreach \y in {0,1,2,3}{

    \ifnum \inteval{\i+\j}>0
    \ifnum \inteval{\k+\l}>0
    
    \ifnum \x=\y
        \ifnum \inteval{(\i-\l)*(\i-\l)+(\j-\k)*(\j-\k)}=1
            \draw[bline] (\i,\j,\x) to (\l,\k,\y);
        \fi
    \fi

    \ifnum \inteval{(\x-\y)*(\x-\y)}=1
        \ifnum \i=\l
            \ifnum \j=\k
                \draw[bline] (\i,\j,\x) to (\l,\k,\y);
            \fi
        \fi
    \fi

    \ifnum \inteval{(\x-\y)*(\x-\y)}=1
        \ifnum \inteval{(\i-\l)*(\i-\l)+(\j-\k)*(\j-\k)}=1
            \draw[line, opacity=0.2] (\i,\j,\x) to (\l,\k,\y);
        \fi
    \fi

    \fi
    \fi
}}}
}}}

% inside (lexicographic)
\foreach \i in {0,1}{
\foreach \j in {0,1}{
\foreach \x in {0,1,2,3}{
\foreach \l in {0,1}{
\foreach \k in {0,1}{
\foreach \y in {0,1,2,3}{
    \ifnum \inteval{(\x-\y)*(\x-\y)}=1
        \ifnum \inteval{(\i-\l)*(\i-\l)+(\j-\k)*(\j-\k)}>1
            \draw[lline, opacity=0.6] (\i,\j,\x) to (\l,\k,\y);
        \fi
    \fi
    
}}}
}}}

\foreach \i in {0,1}{
\foreach \j in {0,1}{
\foreach \x in {0,1,2,3}{
    \ifnum \inteval{\i+\j}>0
        \node[vertex, inner sep=1.3pt] at (\i,\j,\x) {};
    \fi
}}}

% front

\foreach \i in {0,1}{
\foreach \j in {0,1}{
\foreach \x in {0,1,2,3}{
\foreach \l in {0,1}{
\foreach \k in {0,1}{
\foreach \y in {0,1,2,3}{

    \ifnum \inteval{\i+\j}<2
    \ifnum \inteval{\k+\l}<2
    
    \ifnum \x=\y
        \ifnum \inteval{(\i-\l)*(\i-\l)+(\j-\k)*(\j-\k)}=1
            \draw[bline] (\i,\j,\x) to (\l,\k,\y);
        \fi
    \fi

    \ifnum \inteval{(\x-\y)*(\x-\y)}=1
        \ifnum \i=\l
            \ifnum \j=\k
                \draw[bline] (\i,\j,\x) to (\l,\k,\y);
            \fi
        \fi
    \fi

    \ifnum \inteval{(\x-\y)*(\x-\y)}=1
        \ifnum \inteval{(\i-\l)*(\i-\l)+(\j-\k)*(\j-\k)}=1
            \draw[line] (\i,\j,\x) to (\l,\k,\y);
        \fi
    \fi

    \fi
    \fi
}}}
}}}

\foreach \i in {0,1}{
\foreach \j in {0,1}{
\foreach \x in {0,1,2,3}{
    \ifnum \inteval{\i+\j}<2
        \node[vertex] at (\i,\j,\x) {};
    \fi
}}}
\end{tikzpicture}

\begin{tikzpicture}[vertex/.style={circle,draw,fill=black,inner sep=1.5pt, black}, /tikz/xscale=\s, /tikz/yscale=\s]

%%
% Labels
%%

% label
\node at (-2.85,0) {$P_4\cprod C_4$};
\node at (0,0) {$P_4\boxtimes C_4$};
\node at (2.85,0) {$P_4\lex C_4$};
% \node at (-2.75,0) {$P_4\cprod C_4$};
% \node at (0,0) {$P_4\boxtimes C_4$};
% \node at (2.7,0) {$P_4\lex C_4$};
\end{tikzpicture}

\caption{Examples of $P_4$ times $C_4$ using different products. Notice that $P_4\boxtimes C_4$ is $P_4\cprod C_4$ plus the red edges, and $P_4\lex C_4$ is $P_4\boxtimes C_4$ plus the blue edges.
%  The colours highlight the different edges appearing in each product.
}
\label{fig:product-examples}
\end{figure}

Our next result, proven in \cref{sec:treewidth-of-product-graphs}, is a lower bound on the treewidth of a Cartesian product:

\begin{theorem}
\label{thm:main-cartesian-product}
For all graphs $G$ and $H$, $\tw(G\cprod H) \geq \frac{1}{2}(\tw(G)+1)(\tw(H)+1)-1$.
\end{theorem}

The constant $\frac{1}{2}$ in \cref{thm:main-cartesian-product} is best possible since \citet{HARVEY2018157} showed that $\pw(K_m \cprod K_n) = (\frac{1}{2}+o(1))mn$ for $m\geq n$ (see \cite{LUCENA20071055} for the $m=n$ case).

A key tool in the proof of \cref{thm:main-strong-product} is the following dual notion to treewidth introduced by \citet{SEYMOUR199322} called the bramble number.
Two connected subsets $A$ and $B$ of $V(G)$ \define{touch} if $G[A\cup B]$ is connected. A \define{bramble} in $G$ is a collection $\sB$ of non-empty connected subsets of $V(G)$ that pairwise touch. A \define{hitting set} for $\sB$ is a set $S\subseteq V(G)$ that intersects every element of $\sB$. The \define{order} of $\sB$, denoted by \define{$\ord(\sB)$}, is the minimum size of a hitting set for $\sB$. The \define{bramble number} of $G$, denoted by \define{$\bn(G)$}, is the maximum order of a bramble in $G$.

\begin{theorem}[\cite{SEYMOUR199322}]
\label{thm:treewidth-duality}
For every graph $G$, $\tw(G)+1=\bn(G)$.
\end{theorem}

As an aside, \citet{BELLENBAUM_DIESTEL_2002} and \citet{mazoit2013simpleprooftreewidthduality} each gave a short proof of \cref{thm:treewidth-duality}.

Brambles are useful for lower bounding the treewidth. For example, recall that the $n\times n$-grid has treewidth $n$. The lower bound of $n$ follows from \cref{thm:treewidth-duality} by considering the bramble consisting of the top row, the last column minus the top vertex, and the crosses in the bottom-left $(n-1)\times (n-1)$-grid. 
A well-known generalisation of this lower bound says that $\tw(G\cprod H) \geq n$ for all connected graphs $G$ and $H$ with at least $n\geq 2$ vertices (see \cite{dujmovic2025gridminorsproducts} for a proof).
Wood \cite{WoodProductConnected} improved this bound by taking into account the connectivity of $G$ and $H$.\footnote{\cref{thm:Wood} is stated in \cite{WoodProductConnected} as $\tw(G\cprod H) \geq k(n-2k+2)-1$, however it is implicit in the proof that the strict bramble number of $G\cprod H$, which we define later, is at least $k(n-2k+2)$. Then \citep[Lemma 2.6]{AIDUN20201} implies $\tw(G\cprod H) \geq k(n-2k+2)$.}

\begin{theorem}[\cite{WoodProductConnected}]
\label{thm:Wood}
For all $k$-connected graphs $G$ and $H$ with at least $n\geq 2$ vertices,
\[\tw(G\cprod H) \geq k(n-2k+2).\]
\end{theorem}

Qualitatively, the bound in \cref{thm:Wood} is of a very different character from the bounds in \cref{thm:main-strong-product,thm:main-cartesian-product}, which are solely in terms of treewidth. This is because treewidth is a global structural parameter, unlike connectivity, and the number of vertices in a graph can be much larger than its treewidth. Quantitatively, the bound in \cref{thm:Wood} is better 
when the factors have treewidth comparable to their connectivity. On the other hand, \cref{thm:main-strong-product,thm:main-cartesian-product} are superior when the factors have large treewidth and small connectivity, such as random $d$-regular graphs.
An important consequence of \cref{thm:main-strong-product,thm:main-cartesian-product}, which is made precise in \cref{sec:vertex-expansion}, is that products of expanders have large subgraphs that are expanders.

Since the treewidth of a graph is at least its connectivity \cite[Exercise 12.26]{Diestel}, \cref{thm:Wood}
inspires the question: Can the bound $\tw(G\cprod H) \in \Omega(\tw(G)\tw(H))$ in \cref{thm:main-cartesian-product} be improved to $\tw(G\cprod H) \in \Omega(\tw(G)|V(H)|)$?
The answer is `no' since Hickingbotham and Wood \citep[Proposition 21]{HickingbothamStructuralPropertiesofGraphProducts} showed that for all $n\geq k+1$, there exists an $n$-vertex treewidth $k$ graph $G_{k,n}$ such that $\tw(G_{k,n}\boxtimes G_{k,n}) \in \mathcal{O}(n+k^2)$.

On a related note, \citet{HickingbothamStructuralPropertiesofGraphProducts} asked whether $\tw(G\boxtimes H)\in \mathcal{O}(\tw(G\cprod H))$? 
Our next theorem shows that the analogous relationship fails for $\lex$ and $\boxtimes$,
that is $\tw(G\lex H) \not\in \mathcal{O}(\tw(G\boxtimes H))$. More strongly, in \cref{sec:strong-product-vs-lexicographic-product} we show that $\tw(G\boxtimes H)$ and $\tw(G\lex H)$ are not tied.

\begin{theorem}
\label{thm:main-strong-vs-lexicographic}
There is no function $f:\mathbb{N}\rightarrow \mathbb{N}$ such that $\tw(G\lex H) \leq f(\tw(G\boxtimes H))$ for all graphs $G$ and $H$.
\end{theorem}

Guided by the role of brambles in our proof of \cref{thm:main-strong-product}, we next consider the strict bramble number.
A bramble in a graph $G$ is \define{strict} if its elements are pairwise intersecting. The \define{strict bramble number} of $G$, denoted by \define{$\sbn(G)$}, is the maximum order of a strict bramble in $G$.

The next two theorems give strict bramble analogues of our earlier product inequalities. The first is proven in \cref{sec:treewidth-of-product-graphs}:

\begin{theorem}
\label{thm:main-cartesian-sbn-product}
For all graphs $G$ and $H$, $\tw(G\cprod H) \geq \sbn(G)(\tw(H)+1)-1$.
\end{theorem}

\cref{thm:main-cartesian-sbn-product} strengthens the bound $\tw(G\cprod H) \geq \sbn(G)(\had(H)+1)-1$ by \citet*{KOZAWA2014251}, who introduced the notion of strict brambles under the name ``PI family''.\footnote{The name ``strict bramble'' is by \citet*{AIDUN20201}} They also showed that strict bramble number is tied to treewidth, in particular, $\sbn(G) \geq \frac{1}{2}(\tw(G)+1)$. Therefore \cref{thm:main-cartesian-sbn-product} implies \cref{thm:main-cartesian-product}.

A duality result (\cref{thm:sbn-duality} below) for strict brambles analogous to \cref{thm:treewidth-duality} was proved by \citet{lardas2023strict}. In \cref{sec:strict-bramble-number-of-cartesian-products}, we use their duality result to obtain the following lower bound on the strict bramble number of Cartesian products:

\begin{theorem}
\label{thm:main-sbn}
For all graphs $G$ and $H$, $\sbn(G\cprod H) \geq \sbn(G)\sbn(H)$.
\end{theorem}

We now highlight parallels between \cref{thm:main-sbn,thm:main-strong-product}. By \cref{thm:treewidth-duality}, we may rewrite \cref{thm:main-strong-product} as $\bn(G\boxtimes H) \geq \bn(G)\bn(H)$. Therefore \cref{thm:main-strong-product,thm:main-sbn} motivate viewing $(\boxtimes, \bn)$ and $(\cprod, \sbn)$ as natural `product-parameter pairings', in the sense that the parameter respects the product via a multiplicative lower bound. \cref{thm:main-sbn} pertains to the more restrictive product, which necessitates a more restrictive parameter. Indeed, the proofs of \cref{thm:main-strong-product,thm:main-sbn} are nearly identical, differing only in the modifications required to accommodate the more restrictive setting. The pairing is also supported by the optimality of \cref{thm:main-cartesian-product}, specifically, if $\bn$ is used in place of $\sbn$ in \cref{thm:main-sbn}, then a factor of $\frac{1}{2}$ must be introduced.

\section{Treewidth of Product Graphs}
\label{sec:treewidth-of-product-graphs}

This section proves \cref{thm:main-strong-product,thm:main-cartesian-product,thm:main-cartesian-sbn-product} (restated below).

Recall that in \cref{sec:intro} it was explained that \cref{thm:main-cartesian-product} is a consequence of \cref{thm:main-cartesian-sbn-product}. As such, we only prove \cref{thm:main-strong-product,thm:main-cartesian-sbn-product}.

Let $G$ be a graph and let $(T,\beta)$ be a tree-decomposition of $G$. For each set $X\subseteq V(G)$, let $\mathdefine{(T,\beta)\angle{X}}\coloneqq\set{t\in V(T) : X\cap \beta(t)\not=\emptyset}$.

The following are well-known facts about tree-decompositions and brambles (see \cite{Reed97} for proofs):

\begin{fact}
\label{fact:trace-of-connected-sets-td}
For every graph $G$ and every tree-decomposition $(T,\beta)$ of $G$, if $X\subseteq V(G)$ is connected in $G$, then $(T,\beta)\angle{X}$ is connected in $T$.
\end{fact}

\begin{fact}
\label{fact:bag-bramble}
For every bramble $\sB$ and every tree-decomposition of a graph $G$, some bag is a hitting set for $\sB$.
\end{fact}

We now prove \cref{thm:main-strong-product}, which states that for all graphs $G$ and $H$
\[\tw(G\boxtimes H) \geq (\tw(G)+1)(\tw(H)+1)-1.\]

\begin{proof}[Proof of \cref{thm:main-strong-product}]
Let \define{$(T, \beta)$} be a minimum-width tree-decomposition of $G\boxtimes H$ and let \define{$\sB$} be a maximum-order bramble in $G$. For each vertex $v\in V(H)$, let $\mathdefine{\sB_v}\coloneqq\set{A\times \set{v} : A\in \sB}$. Then each $\sB_v$ is a bramble in $G\boxtimes H$ and $\ord(\sB_v)=\ord(\sB)=\bn(G)$. For each node $t\in V(T)$, let
\[\mathdefine{\beta'(t)} \coloneqq \set{v\in V(H) : \beta(t) \text{ is a hitting set for }\sB_v}.\]

We show that $(T,\beta')$ is a tree-decomposition of $H$. Consider any vertex $v\in V(H)$. Since $\sB_v$ is a bramble in $G\boxtimes H$ and $(T,\beta)$ is a tree-decomposition of $G\boxtimes H$, \cref{fact:bag-bramble} implies that there exists a node $x\in V(T)$ such that $\beta(x)$ is a hitting set for $\sB_v$, which implies $v\in \beta'(x)$. Suppose there exists a node $y\in V(T-x)$ such that $v\in \beta'(y)$. Then for all $A\times \set{v}\in \sB_v$, since $A\times \set{v}$ is connected and intersects both $\beta(x)$ and $\beta(y)$, \cref{fact:trace-of-connected-sets-td} implies that $A\times \set{v}$ intersects $\beta(t)$ for every node $t \in V(xTy)$. Consequently, $\beta(t)$ is a hitting set for $\sB_v$ for all $t\in V(xTy)$. Thus $\set{t\in V(T) : v\in \beta'(t)}$ is non-empty and connected in $T$, and $(T,\beta')$ has the vertex-property. 
Next, consider any edge $uv\in E(H)$. The following shows that $\sB_u\cup \sB_v$ is a bramble in $G\boxtimes H$: Since $\sB_u$ and $\sB_v$ are brambles in $G\boxtimes H$, it suffices to show that every $U\times \set{u} \in \sB_u$ and every $V\times \set{v} \in \sB_v$ touch in $G\boxtimes H$. Since $\sB$ is a bramble in $G$, $U$ and $V$ touch in $G$. If $U\cap V=\emptyset$, then there exists an edge $pq\in E(G)$ with $p\in U$ and $q\in V$, therefore $(p,u)(q,v)$ is an edge between $U\times \set{u}$ and $V\times \set{v}$. On the other hand, there exists $w\in U\cap V$. Thus $(w,u)(w,v)$ is an edge between $U\times \set{u}$ and $V\times \set{v}$. In either case, $U\times \set{u} \in \sB_u$ and $V\times \set{v} \in \sB_v$ touch in $G\boxtimes H$, hence $\sB_u\cup \sB_v$ is a bramble in $G\boxtimes H$. Then by \cref{fact:bag-bramble}, there exists a node $t\in V(T)$ such that $\beta(t)$ is a hitting set for $\sB_u\cup \sB_v$. Therefore $\set{u,v}\subseteq \beta'(t)$, and $(T,\beta')$ has the edge-property. Hence $(T,\beta')$ is a tree-decomposition of $H$ as claimed. Fix $t\in V(T)$ such that $|\beta'(t)|$ is maximum. So $|\beta'(t)| \geq \tw(H)+1$. Then
\[\tw(G\boxtimes H)+1 \geq |\beta(t)| \geq \sum_{v\in \beta'(t)}|\beta(t)\cap (V(G)\times \set{v})| \geq \sum_{v\in \beta'(t)}\ord(\sB_v) \geq \bn(G)(\tw(H)+1).\]
Therefore $\tw(G\boxtimes H)+1 \geq (\tw(G)+1)(\tw(H)+1)$ by \cref{thm:treewidth-duality}.
\end{proof}

Next we prove \cref{thm:main-cartesian-sbn-product}, which states that for all graphs $G$ and $H$
\[\tw(G\cprod H) \geq \sbn(G)(\tw(H)+1)-1.\]
We remark that our proof of \cref{thm:main-cartesian-sbn-product} follows the same method as \cref{thm:main-strong-product}, except we use strict brambles.

\begin{proof}[Proof of \cref{thm:main-cartesian-sbn-product}]
Let \define{$(T, \beta)$} be a minimum-width tree-decomposition of $G\cprod H$ and let \define{$\sS$} be a maximum-order strict bramble in $G$. For each vertex $v\in V(H)$, let $\mathdefine{\sS_v}\coloneqq\set{A\times \set{v} : A\in \sS}$. Then each $\sS_v$ is a strict bramble in $G\cprod H$ and $\ord(\sS_v)=\ord(\sS)=\sbn(G)$. For each node $t\in V(T)$, let 
\[\mathdefine{\beta'(t)} \coloneqq \set{v\in V(H) : \beta(t) \text{ is a hitting set for }\sS_v}.\]

We show that $(T,\beta')$ is a tree-decomposition of $H$. Consider any vertex $v\in V(H)$. Since $S_v$ is a bramble in $G\cprod H$ and $(T,\beta)$ is a tree-decomposition of $G\cprod H$, \cref{fact:bag-bramble} implies that there exists a node $x\in V(T)$ such that $\beta(x)$ is a hitting set for $\sS_v$, which implies $v\in \beta'(x)$. Suppose there exists a node $y\in V(T-x)$ such that $v\in \beta'(y)$. Then for all $A\times\set{v}\in \sS_v$, since $A\times\set{v}$ is connected and intersects both $\beta(x)$ and $\beta(y)$, \cref{fact:trace-of-connected-sets-td} implies that $A\times\set{v}$ intersects $\beta(t)$ for every node $t\in V(xTy)$. Consequently, $\beta(t)$ is a hitting set for $\sS_v$ for all $t\in V(xTy)$. Thus $\set{t\in V(T) : v\in \beta'(t)}$ is non-empty and connected in $T$, and $(T,\beta')$ has the vertex-property. Next, consider any edge $uv\in E(H)$. The following shows that $\sS_u\cup \sS_v$ is a bramble in $G\cprod H$: Since $\sS_u$ and $\sS_v$ are brambles in $G\cprod H$, it suffices to show that every $U\times \set{u}\in \sS_u$ and every $V\times \set{v}\in \sS_v$ touch in $G\cprod H$. Since $\sS$ is a strict bramble in $G$, there exists $w\in U\cap V$, which implies that $(w, u)(w, v)$ is an edge between $U\times \set{u}$ and $V\times \set{v}$. Therefore $U\times \set{u}$ and $V\times \set{v}$ touch in $G\cprod H$, and $\sS_u\cup \sS_v$ is a bramble in $G\cprod H$. Then by \cref{fact:bag-bramble}, there exists a node $t\in V(T)$ such that $\beta(t)$ is a hitting set for $\sS_u\cup \sS_v$, so $\set{u,v}\subseteq \beta'(t)$, and $(T,\beta')$ has the edge-property. Hence $(T,\beta')$ is a tree-decomposition of $H$ as claimed. Fix $t\in V(T)$ such that $|\beta'(t)|$ is maximum. So $|\beta'(t)| \geq \tw(H)+1$. Then
\[\tw(G\cprod H)+1 \geq |\beta(t)| \geq \! \sum_{v\in \beta'(t)}|\beta(t)\cap (V(G)\times \set{v})| \geq \! \sum_{v\in \beta'(t)}\ord(\sS_v) \geq \sbn(G)(\tw(H)+1),\]
as required.
\end{proof}

In our proof of \cref{thm:main-cartesian-sbn-product}, the underlying tree $T$ in the tree-decomposition of $G\cprod H$ also indexes the final tree-decomposition of $H$. Hence if we instead take $(T,\beta)$ to be a minimum-width path-decomposition of $G\cprod H$, we would conclude $\pw(G\cprod H) \geq \sbn(G)(\pw(H)+1)-1$. Since $\sbn(G)\geq \frac{1}{2}(\tw(G)+1)$, we deduce
\begin{corollary}
\label{cor:pathwidth-cartesian-product}
For all graphs $G$ and $H$, $\pw(G\cprod H) \geq \frac{1}{2}(\tw(G)+1)(\pw(H)+1)-1$.
\end{corollary}

\section{Pathwidth of a Strong Product}
\label{sec:pathwidth-of-strong-product}

This section proves \cref{thm:main-strong-product-pw}, which states that for all graphs $G$ and $H$
\[\pw(G\boxtimes H) \geq (\pw(G)+1)(\pw(H)+1)-1.\]

In their proof of the Tree-Minor Theorem,\footnote{The Tree-Minor Theorem states that if a forest $F$ is not a minor of $G$, then $\pw(G)\leq |V(F)|-2$.} \citet{BIENSTOCK1991274} introduced a dual parameter for pathwidth, called blockages. The first step in their proof is a duality result (\cref{thm:pathwidth-duality} below) that is analogous to \cref{thm:treewidth-duality}. We use their duality result in our proof of \cref{thm:main-strong-product-pw}.

Let $G$ be a graph and let $k$ be a non-negative integer. For each set $X\subseteq V(G)$, let \define{$\partial X$} be the set of vertices in $X$ that have a neighbour in $V(G)\setminus X$. A \define{$k$-cut} in $G$ is a pair $(A,B)$ of subsets of $V(G)$ such that $|A\cap B| \leq k$ and $G[A]\cup G[B]=G$. A \define{cut} is a $k$-cut for some $k$. A \define{blockage of order $k$} in $G$ is a set $\sB\subseteq 2^{V(G)}$ that satisfies the following \define{blockage axioms}:
\begin{itemize}
    \item if $X\in \sB$, then $|\partial X| \leq k$;
    \item if $Y\subseteq X \in \sB$ and $|\partial Y| \leq k$, then $Y\in \sB$;
    \item if $(X,Y)$ is a $k$-cut in $G$, then $\sB$ contains exactly one of $X$ or $Y$.
\end{itemize}

\begin{theorem}[\cite{BIENSTOCK1991274}]
\label{thm:pathwidth-duality}
$G$ has a blockage of order $k$ if and only if $\pw(G)\geq k$.
\end{theorem}

Subsequent proofs of the Tree-Minor Theorem by \citet{diestel1995graph} and \citet{seymour2023shorter} avoid pathwidth-blockage duality by employing a shorter and more direct approach.

For convenience, our proof of \cref{thm:main-strong-product-pw} uses the notion of stoppages, which were also introduced by \citet[\S 4]{BIENSTOCK1991274} and are essentially equivalent to blockages. Let $G$ be a graph and let $k$ be a non-negative integer. A \define{stoppage of order $k$} in $G$ is a set $\mathscr{Y}$ of $k$-cuts in $G$ that satisfies the following \define{stoppage axioms}:
\begin{itemize}
    \item if $(A,B)$ is a $k$-cut in $G$, then $(A,B)\in \mathscr{Y}$ or $(B,A)\in \mathscr{Y}$;
    \item if $(A,B)\in \mathscr{Y}$ and $(C,D)\in \mathscr{Y}$, then $G[A]\cup G[C]\not=G$.
\end{itemize}

\citet{BIENSTOCK1991274} stated the relationship between blockages and stoppages but left the details to the reader. We provide the details for completeness.

\begin{lemma}[{\citep[4.1 (reworded)]{BIENSTOCK1991274}}]
\label{lem:blockage-stoppage}
$G$ has a blockage of order $k$ if and only if $G$ has a stoppage of order $k$.
\end{lemma}

\begin{proof}
$(\implies)$ Suppose $\sB$ is a blockage of order $k$ in $G$. Let $\mathscr{Y}$ be the set of all $k$-cuts $(A,B)$ with $A\in \sB$. We show that $\mathscr{Y}$ is a stoppage of order $k$ in $G$. The first stoppage axiom follows from the third blockage axiom. For the second stoppage axiom, suppose there exists $(A,B)\in \mathscr{Y}$ and $(C,D)\in \mathscr{Y}$ with $G[A]\cup G[C]=G$. Let $X$ be the set of vertices in $A\cap B$ that have no neighbour in $B\setminus A$, and let $B'\coloneqq B\setminus X$. Then $(A, B')$ is a $k$-cut. Since $A\in \sB$, the third blockage axiom implies $B'\not\in \sB$. Note that $B'\setminus A\subseteq C$ because $G[A]\cup G[C]=G$. Moreover, each vertex $u\in A\cap B'$ has a neighbour $v\in B'\setminus A$. Then since $uv\not\in E(G[A])$, $uv\in E(G[C])$, thus $u\in C$. Hence $B'\subseteq C$, so the second blockage axiom implies $B'\in \sB$, contradiction.

$(\impliedby)$ Suppose $\mathscr{Y}$ is a stoppage of order $k$ in $G$. Let $\sB \coloneqq \set{A : (A,B)\in \mathscr{Y}}$. We show that $\sB$ is a blockage of order $k$ in $G$. Since every element of $\mathscr{Y}$ is a $k$-cut, the first blockage axiom holds. For the second blockage axiom, suppose $Y\subseteq X\in \sB$ and $|\partial Y| \leq k$. Then there exists $W\subseteq V(G)$ such that $(X,W)\in \mathscr{Y}$. Furthermore, by letting $Z\coloneqq (V(G)\setminus Y)\cup \partial Y$, we have that $(Y, Z)$ is a $k$-cut in $G$. Since $Y\subseteq X$, it follows that $(V(G)\setminus X)\cup \partial X \subseteq (V(G)\setminus Y)\cup \partial Y$, thus $G[X]\cup G[Z] = G$. Hence by the second stoppage axiom, $(Z,Y)\not\in \mathscr{Y}$, thus $(Y,Z)\in \mathscr{Y}$ by the first stoppage axiom, so $Y\in \sB$. Next we prove the third blockage axiom. Suppose that $(X,Y)$ is a $k$-cut in $G$. Then the first stoppage axiom implies $X\in \sB$ or $Y\in \sB$. If $X \in \sB$ and $Y\in \sB$, then $(X,W)\in \mathscr{Y}$ and $(Y,Z)\in \mathscr{Y}$ for some $W,Z\subseteq V(G)$. However since $(X,Y)$ is a cut, $G[X]\cup G[Y]=G$, which contradicts the second stoppage axiom. 
\end{proof}

We now prove \cref{thm:main-strong-product-pw}.

\begin{proof}[Proof of \cref{thm:main-strong-product-pw}]
Let $\mathdefine{k}\coloneqq(\pw(G)+1)(\pw(H)+1)-1$. Let \define{$\mathscr{G}$} be a stoppage of order $\pw(G)$ in $G$, and let \define{$\mathscr{H}$} be a stoppage of order $\pw(H)$ in $H$. For each $u\in V(H)$, let 
\[\mathdefine{G_u}\coloneqq(G\boxtimes H)[V(G)\times \set{u}] \quad \text{ and } \quad \mathdefine{\mathscr{G}_u}\coloneqq\set{(A\times \set{u}, B\times\set{u}) : (A,B)\in \mathscr{G}}.\]
Then $\mathscr{G}_u$ is a stoppage of order $\pw(G)$ in $G_u$. For each $u\in V(H)$ and each subset $S\subseteq V(G\boxtimes H)$, let $\mathdefine{S_u}\coloneqq V(G_u)\cap S$ and let $\mathdefine{\pi(S)}\coloneqq\set{v\in V(G) : (v,w) \in S}$.

Consider each $k$-cut $(A,B)$ in $G\boxtimes H$. Note that $(A_u, B_u)$ is a cut in $G_u$ for each $u\in V(H)$. Hence we may define
\[\mathdefine{X_{(A,B)}} \coloneqq\set{u\in V(H) : |A_u\cap B_u| > \pw(G)\text{ or }(A_u, B_u) \in \mathscr{G}_u}.\]
Note that the first stoppage axiom for $\mathscr{G}_u$ implies:
\begin{equation}
\label{eq:side-of-X}
(A_u,B_u)\in \mathscr{G}_u \text{ for all } u\in X_{(A, B)}\setminus X_{(B, A)}.\tag{$\star$}
\end{equation}

\begin{claim}
\label{clm:cut-transfer}
If $(A,B)$ is a $k$-cut in $G\boxtimes H$, then $(X_{(A, B)}, X_{(B, A)})$ is a $\pw(H)$-cut in $H$.
\end{claim}

\begin{subproof}
We first show that $X_{(A, B)} \cup X_{(B,A)} = V(H)$. Consider any $u\in V(H) \setminus X_{(A, B)}$. Then $|A_u\cap B_u| \leq \pw(G)$ and $(A_u, B_u) \not\in \mathscr{G}_u$, therefore by the first stoppage axiom for $\mathscr{G}_u$, $(B_u, A_u) \in \mathscr{G}_u$, implying that $u\in X_{(B, A)}$. Hence $X_{(A, B)} \cup X_{(B, A)} = V(H)$.

By the stoppage axioms, $X_{(A, B)}\cap X_{(B, A)} = \set{u\in V(H) : |A_u\cap B_u|>\pw(G)}$, thus
\[(\pw(G)+1)|X_{(A, B)}\cap X_{(B, A)}| \leq \sum_{u\in X_{(A, B)}\cap X_{(B, A)}}|A_u\cap B_u| \leq |A\cap B| \leq k,\]
which implies $|X_{(A, B)}\cap X_{(B, A)}|\leq \pw(H)$.

Finally, we show that there is no edge $uv\in E(H)$ such that $u\in X_{(A,B)} \setminus X_{(B,A)}$ and $v\in X_{(B,A)} \setminus X_{(A,B)}$. Assume for a contradiction there is such an edge $uv\in E(H)$. Then \eqref{eq:side-of-X} implies that $(A_u, B_u) \in \mathscr{G}_u$ and $(B_v, A_v) \in \mathscr{G}_v$. Therefore $(\pi(A_u),\pi(B_u))\in \mathscr{G}$ and $(\pi(B_v), \pi(A_v))\in \mathscr{G}$. Then by the second stoppage axiom for $\mathscr{G}$, $G[\pi(A_u)]\cup G[\pi(B_v)]\not= G$. We now show that $(\pi(A_u), \pi(B_v))$ is a cut in $G$, which will contradict the previous sentence. As a first step, we show that $\pi(A_u)\cup \pi(B_v)=V(G)$. Consider any $b\in V(G)\setminus \pi(A_u)$. Thus $(b,u)\not\in A_u$, implying that $(b,u)\not\in A$. Therefore since $(A,B)$ is a cut in $G\boxtimes H$, $(b,u)\in B\setminus A$. Then $(b,v) \in N_{G\boxtimes H}(b,u)\subseteq B$, which implies $b\in \pi(B_v)$. Hence $\pi(A_u)\cup \pi(B_v)=V(G)$. Next, assume for a contradiction that there exists an edge $ab\in E(G)$ such that $a\in \pi(A_u)\setminus \pi(B_v)$ and $b\in \pi(B_v)\setminus \pi(A_u)$. Since $a\not\in \pi(B_v)$, $(a,v)\not\in B_v$, implying that $(a,v)\not\in B$. Therefore since $(A,B)$ is a cut in $G\boxtimes H$, $(a,v)\in A\setminus B$. Similarly, from $b\not\in \pi(A_u)$ we deduce that $(b,u)\in B\setminus A$. Now since $ab\in E(G)$ and $uv\in E(H)$, $(a,v)(b,u)\in E(G\boxtimes H)$, which along with $(a,v)\in A\setminus B$ and $(b,u)\in B\setminus A$ contradicts $(A,B)$ being a cut in $G\boxtimes H$. Hence we have proved that $(\pi(A_u), \pi(B_v))$ is a cut in $G$, so $G[\pi(A_u)]\cup G[\pi(B_v)]= G$, the desired contradiction.
\end{subproof}

Let \define{$\mathscr{Y}$} be the set of all $k$-cuts $(A,B)$ in $G\boxtimes H$ such that $(X_{(A,B)}, X_{(B,A)})\not\in \mathscr{H}$.

\begin{claim}
\label{clm:Y-stoppage}
$\mathscr{Y}$ is a stoppage of order $k$ in $G\boxtimes H$.
\end{claim}

\begin{subproof}
We begin by proving the first stoppage axiom for $\mathscr{Y}$. Suppose there exists a $k$-cut $(A,B)$ in $G\boxtimes H$ such that $(A,B)\not\in \mathscr{Y}$ and $(B,A)\not\in \mathscr{Y}$.~Then $(X_{(A,B)}, X_{(B,A)}) \in \mathscr{H}$ and $(X_{(B,A)}, X_{(A,B)})\in \mathscr{H}$. Therefore $(X_{(A,B)}, X_{(B,A)})$ is a cut in $H$, and $H[X_{(A,B)}]\cup H[X_{(B,A)}] \not= H$ by the second stoppage axiom for $\mathscr{H}$, a contradiction. Hence $(A,B)\in \mathscr{Y}$ or $(B,A)\in \mathscr{Y}$, as required.

We now prove the second stoppage axiom for $\mathscr{Y}$. Consider any $(A,B)\in \mathscr{Y}$ and any $(C,D)\in \mathscr{Y}$. By definition of $\mathscr{Y}$, $(X_{(A,B)}, X_{(B,A)})\not\in \mathscr{H}$ and $(X_{(C,D)}, X_{(D,C)})\not\in \mathscr{H}$. Note that by \cref{clm:cut-transfer}, $(X_{(A,B)}, X_{(B,A)})$ and $(X_{(C,D)}, X_{(D,C)})$ are $\pw(H)$-cuts in $H$, thus the first stoppage axiom for $\mathscr{H}$ implies that $(X_{(B,A)}, X_{(A,B)}) \in \mathscr{H}$ and $(X_{(D,C)}, X_{(C,D)}) \in \mathscr{H}$. Then the second stoppage axiom for $\mathscr{H}$ implies $H[X_{(B,A)}] \cup H[X_{(D,C)}] \not= H$. Hence there exists $u\in X_{(A,B)}\setminus X_{(B,A)}$ and $v\in X_{(C,D)}\setminus X_{(D,C)}$ such that $\set{u}$ and $\set{v}$ touch in $H$. Moreover, \eqref{eq:side-of-X} implies that $(A_u, B_u) \in \mathscr{G}_u$ and $(C_v, D_v)\in \mathscr{G}_v$. Then $(\pi(A_u), \pi(B_u))\in \mathscr{G}$ and $(\pi(C_v), \pi(D_v))\in \mathscr{G}$, so $G[\pi(A_u)]\cup G[\pi(C_v)]\not= G$ by the second stoppage axiom for $\mathscr{G}$. Hence there exists $b\in \pi(B_u)\setminus \pi(A_u)$ and $d\in \pi(D_v)\setminus\pi(C_v)$ such that $\set{b}$ and $\set{d}$ touch in $G$. Consequently, $\set{(b,u)}$ and $\set{(d,v)}$ touch in $G\boxtimes H$. To complete the argument, we show that $(b,u)\in B\setminus A$ and $(d,v)\in D\setminus C$, which will imply that $(G\boxtimes H)[A]\cup (G\boxtimes H)[C]\not= G\boxtimes H$. Indeed $(b,u)\in B_u \subseteq B$, and if $(b,u)\in A$, then $b\in \pi(A_u)$, contrary to the definition of $b$. Similarly we deduce that $(d,v)\in D\setminus C$, as desired.
\end{subproof}

Therefore \cref{thm:pathwidth-duality,lem:blockage-stoppage,clm:Y-stoppage} imply that $\pw(G\boxtimes H) \geq k$.
\end{proof}

\section{Application to Vertex Expansion}
\label{sec:vertex-expansion}

This section shows that products of graph classes with positive vertex expansion contain linearly-sized graphs with positive vertex expansion (\cref{thm:main-expansion} below).

Let $G$ be a graph. For a set $X\subseteq V(G)$, let $\mathdefine{N_G(X)} \coloneqq \bigcup_{x\in X}N_G(x) \setminus X$. Define the \define{vertex expansion} of $G$ to be
\[\mathdefine{\varphi(G)} \coloneqq \; \min_{X\subseteq V(G)\, : \,0<|X|\leq \frac{1}{2}|V(G)|}\;\; \frac{|N_G(X)|}{|X|}.\]
Define the \define{vertex expansion} of a graph class $\mathcal{G}$ to be $\mathdefine{\varphi(\mathcal{G})} \coloneqq \inf_{G\in \mathcal{G}}\varphi(G)$.

Grohe and Marx \cite{GROHE2009218} established the following relationships between treewidth and vertex expansion:

\begin{theorem}[{\citep[Proposition 1]{GROHE2009218}}]
\label{thm:expansion-treewidth}
For every graph $G$, $\tw(G)\geq \floor{\tfrac{1}{4}\varphi(G) |V(G)|}$.
\end{theorem}

\begin{theorem}[{\citep[Proposition 2]{GROHE2009218}}]
\label{thm:expansion-subgraph}
Let $n$ be a positive integer and let $\epsilon > 0$. For every $n$-vertex graph $G$ with $\tw(G) \geq \epsilon n$, there exists a subgraph $S\subseteq G$ such that $\tw(S) \geq \tfrac{1}{2} \epsilon n$ and $\varphi(S) \geq \tfrac{1}{2}\epsilon$.
\end{theorem}

For $*\in\set{\cprod, \boxtimes, \lex}$, define $\mathdefine{\mathcal{G} * \mathcal{H}} \coloneqq \set{G*H : G\in\mathcal{G}, H\in \mathcal{H}}$.

\begin{theorem}
\label{thm:main-expansion}
For all $*\in \set{\cprod, \boxtimes, \lex}$, if $\mathcal{G}$ and $\mathcal{H}$ are graph classes with positive vertex expansion, then there exists a constant $c>0$ such that every graph $G * H\in \mathcal{G} * \mathcal{H}$ has a subgraph $S$ with $\tw(S) \geq c |V(G*H)|$, hence $|V(S)|\geq c |V(G*H)|-1$, and $\varphi(S) \geq c$.
\end{theorem}

\begin{proof}
Let $\epsilon_0\coloneqq\min\set{\varphi(\mathcal{G}), \varphi(\mathcal{H})}$ and let $\mathcal{F}\coloneqq\set{F\in \mathcal{G}*\mathcal{H} : \frac{1}{64}\epsilon_0^2|V(F)| < 1}$. Then $\mathcal{F}$ is finite. Since $\varphi(\mathcal{G}) > 0$, every graph in $\mathcal{G}$ has an edge, thus $\epsilon_1\coloneqq\min\set{\tw(F)/|V(F)|:F \in \mathcal{F}}$ is positive. Let $\epsilon\coloneqq\min\set{\tfrac{1}{64}\epsilon_0^2, \epsilon_1}$. Now consider each $G* H\in \mathcal{G}* \mathcal{H}$. By \cref{thm:main-cartesian-product} and \cref{thm:expansion-treewidth},
\begin{align*}
    \tw(G* H) \geq \tw(G \cprod H) \geq \tfrac{1}{2}(\tw(G)+1)(\tw(H)+1)-1 &\geq \tfrac{1}{2}(\tfrac{1}{4} \epsilon_0 |V(G)|)(\tfrac{1}{4} \epsilon_0 |V(H)|)-1\\
    &= \tfrac{1}{32} \epsilon_0^2 |V(G* H)| - 1.
\end{align*}
If $G* H\not\in \mathcal{F}$, then $\tw(G* H) \geq \tfrac{1}{32}\epsilon_0^2 |V(G* H)| - 1 \geq \tfrac{1}{64}\epsilon_0^2 |V(G* H)|$. On the other hand if $G* H\in \mathcal{F}$, then $\tw(G* H) \geq \epsilon_1 |V(G* H)|$. Hence $\tw(G* H) \geq \epsilon |V(G* H)|$ for all $G* H \in \mathcal{G}* \mathcal{H}$. Therefore \cref{thm:expansion-subgraph} implies the result with $c\coloneqq\tfrac{1}{2} \epsilon$.
\end{proof}

\section{Strong Product Versus Lexicographic Product}
\label{sec:strong-product-vs-lexicographic-product}

This section proves \cref{thm:main-strong-vs-lexicographic}, which states that there is no function $f:\mathbb{N}\rightarrow \mathbb{N}$ such that $\tw(G\lex H) \leq f(\tw(G\boxtimes H))$ for all graphs $G$ and $H$.

We begin with the following lemma:

\begin{lemma}
\label{lem:lexicographic}
For all graphs $G$ and $H$, if $\sB$ is a bramble in $G$ such that $|A|\geq 2$ for every $A\in \sB$, then $\tw(G\lex H) \geq \ord(\sB)|V(H)|-1$.
\end{lemma}

\begin{proof}
For each $v\in V(G)$, let $\mathdefine{H_v}\coloneqq (G\lex H)[\set{v}\times V(H)]$. For every subset $X\subseteq V(G)$ and every function $\sigma:X\rightarrow V(G\lex H)$, if $\sigma(v)\in V(H_v)$ for each $v\in X$, then we say that $\sigma$ \define{represents} $X$. Observe that if $uv$ is an edge in $G[X]$, then $\sigma(u)\sigma(v)$ is an edge in $G\lex H$. Therefore if $X$ is connected in $G$, then $\sigma(X)$ is connected in $G\lex H$. For each $X\in \sB$, let $\mathdefine{\mathscr{R}_X} \coloneqq \set{\sigma(X) : \sigma\text{ represents }X}$. Then $\mathdefine{\sB'}\coloneqq\bigcup_{X\in \sB}\mathscr{R}_X$ is a family of connected sets of vertices in $G\lex H$. We claim that $\sB'$ is a bramble in $G\lex H$. To see this, consider any $X', Y'\in \sB'$. Then there exists $X, Y\in \sB$ and there exists $\sigma$ and $\tau$ that represent $X$ and $Y$ respectively such that $X'=\sigma(X)$ and $Y'=\tau(Y)$. Since $\sB$ is a bramble in $G$ and $|X|\geq 2$ and $|Y|\geq 2$, there exists an edge $xy\in E(G)$ such that $x\in X$ and $y\in Y$. Since $\sigma(x)\in V(H_x)$ and $\tau(y)\in V(H_y)$, $\sigma(x)\tau(y)$ is an edge in $G\lex H$ between $X'$ and $Y'$. Hence $X'$ and $Y'$ touch, and $\sB'$ is a bramble in $G\lex H$ as claimed.

Let $S$ be a minimum-sized hitting set for $\sB'$, and let $F\coloneqq\set{v\in V(G) : V(H_v)\cap S \not=\emptyset}$. We claim that if $v\in F$, then $V(H_v)\subseteq S$. Indeed for if not, then there exists $(v,u)\in V(H_v)\cap S$ and $(v,w)\in V(H_v)\setminus S$. By choice of $S$, there exists $X'\in \sB'$ such that $S\cap X'=\set{(v,u)}$. Then there exists $X\subseteq V(G)$ and $\sigma$ that represents $X$ such that $X'=\sigma(X)$. Define $\tau:X\rightarrow V(G\lex H)$ by $\tau(v)=(v,w)$ and $\tau(x)=\sigma(x)$ for all $x\not= v$. Then $\tau$ represents $X$, so $\tau(X)\in \sB'$. However since $(v,w)\not\in S$, $\tau(X)\cap S=\emptyset$, which implies that $S$ is not a hitting set for $\sB'$, a contradiction. It follows that $\bigcup_{v\in F}V(H_v)=S$. Next, assume for a contradiction that $|S| < \ord(\sB)|V(H)|$. Then,
\[|F||V(H)| = \sum_{v\in F}|V(H_v)| = |S| < \ord(\sB)|V(H)|,\]
implying that $|F|\leq \ord(\sB)-1$. Then $F$ is not a hitting set for $\sB$, so there exists $X\in \sB$ such that $X\cap F=\emptyset$, which implies that every element of $\mathscr{R}_X$ is disjoint from $S$, a contradiction. Hence $\bn(G\lex H)\geq \ord(\sB') = |S| \geq \ord(\sB)|V(H)|$. Then by \cref{thm:treewidth-duality}, $\tw(G\lex H) \geq \ord(\sB)|V(H)|-1$.
\end{proof}

As an aside, we point out the following corollary of \cref{lem:lexicographic} that may be of independent interest.

\begin{corollary}
For all graphs $G$ and $H$, if $G$ has a maximum-order bramble with no singleton sets, then $\tw(G\circ H)=(\tw(G)+1)|V(H)|-1$.
\end{corollary}

We now prove \cref{thm:main-strong-vs-lexicographic}.

\begin{proof}[Proof of \cref{thm:main-strong-vs-lexicographic}]
The edges of $K_3$ form a strict bramble of order $2$. Therefore $\tw(K_3\lex P_n) \geq 2n-1$ by \cref{lem:lexicographic}. On the other hand for each $n\geq 2$, $\tw(K_3\boxtimes P_n) \leq (\tw(P_n)+1)|V(K_3)|-1 = 5$. Consequently, there is no function $f:\mathbb{N}\rightarrow \mathbb{N}$ such that $\tw(K_3\lex P_n) \leq f(\tw(K_3\boxtimes P_n))$.
\end{proof}

\section{Strict Bramble Number of Cartesian Products}
\label{sec:strict-bramble-number-of-cartesian-products}

This section proves \cref{thm:main-sbn}, which states that for all graphs $G$ and $H$
\[\sbn(G\cprod H) \geq \sbn(G)\sbn(H).\]

The proof of \cref{thm:main-sbn} is essentially the same as the proof of \cref{thm:main-strong-product}. The main difference is that we replace \cref{thm:treewidth-duality} with the following duality for strict bramble number by \citet*{lardas2023strict}.
A \define{lenient tree-decomposition} of a graph $G$ is a pair $(T,\ell)$ consisting of a tree $T$ and a function $\ell:V(T)\rightarrow 2^{V(G)}$ with the following two properties: (\define{vertex-property}) for every vertex $v\in V(G)$, the set $\set{t\in V(T) : v\in \ell(t)}$ is non-empty and connected in $T$; and (\define{lenient-property}) for every edge $uv\in E(G)$, the sets $\set{t\in V(T) : u\in \ell(t)}$ and $\set{t\in V(T) : v\in \ell(t)}$ touch in $T$. The vertices of $T$ are called \define{nodes}, and the sets $\ell(t)$ are called \define{bags}. The \define{width} of $(T,\ell)$ is $\max_{t\in V(T)} |\ell(t)|$. Note that a lenient tree-decomposition may not be a tree-decomposition.

\begin{theorem}[\cite{lardas2023strict}]
\label{thm:sbn-duality}
For every graph $G$, the minimum width of a lenient tree-decomposition of $G$ equals $\sbn(G)$.
\end{theorem}

We define analogous notation from \cref{sec:treewidth-of-product-graphs} for lenient tree-decompositions. Let $G$ be a graph and let $(T,\ell)$ be a lenient tree-decomposition of $G$. For each set $X\subseteq V(G)$, let $\mathdefine{(T,\ell)\angle{X}}\coloneqq\set{t\in V(T) : X\cap \ell(t)\not=\emptyset}$.

The following basic facts about lenient tree-decompositions and strict brambles have proofs that are similar to \cref{fact:trace-of-connected-sets-td,fact:bag-bramble}:

\begin{fact}
\label{fact:trace-of-connected-sets-ltd}
For every graph $G$ and every lenient tree-decomposition $(T,\ell)$ of $G$, if $X\subseteq V(G)$ is connected in $G$, then $(T,\ell)\angle{X}$ is connected in $T$.
\end{fact}

\begin{comment}
\begin{proof}[Proof of Fact]
Proceed by induction on $|X|$. The base case of $|X|=1$ holds by axioms. Now suppose $|X|\geq 2$ and the result holds for smaller values of $|X|$. Let $u$ be a leaf of a spanning tree of $G[X]$ and let $v\in X\cap N_G(u)$. Then $X\setminus \set{u}$ and $\set{u}$ are connected in $G$. By induction $(T,\ell)\angle{X\setminus \set{u}}$ and $(T,\ell)\angle{u}$ are connected in $T$. Then the lenient-property implies that $(T,\ell)\angle{u}$ and $(T,\ell)\angle{v}$ touch, hence $(T,\ell)\angle{v} \subseteq (T,\ell)\angle{X\setminus \set{u}}$ implies that $(T,\ell)\angle{u}$ and $(T,\ell)\angle{X\setminus \set{u}}$ touch. Therefore $(T,\ell)\angle{X} = (T,\ell)\angle{X\setminus \set{u}} \cup (T,\ell)\angle{u}$ is connected in $T$, as desired.
\end{proof}
\end{comment}

\begin{fact}
\label{fact:lenient-bag-strict-bramble}
For every strict bramble $\sS$ and every lenient tree-decomposition of a graph $G$, some bag is a hitting set for $\sS$.
\end{fact}

\begin{comment}
\begin{proof}[Proof of Fact]
Let $(T,\ell)$ be a lenient tree-decompositions of $G$. Since $\sS$ is a strict bramble, $\set{T[(T,\ell)\angle{A}]:A\in \sS}$ is a family of pairwise intersecting subtrees of $T$. Then by the Helly property there exists $t\in \bigcap_{A\in S}(T,\ell)\angle{A}$. Therefore $\ell(t)$ is a hitting set for $\sS$, as desired.
\end{proof}
\end{comment}

We now prove \cref{thm:main-sbn}.

\begin{proof}[Proof of \cref{thm:main-sbn}]
Let \define{$(T, \ell)$} be a minimum-width lenient tree-decomposition of $G\cprod H$. Then \cref{thm:sbn-duality} implies $\sbn(G\cprod H) \geq |\ell(t)|$ for all $t\in V(T)$. Let \define{$\sS$} be a maximum-order strict bramble in $G$. For each vertex $v\in V(H)$, let $\mathdefine{\sS_v}\coloneqq\set{A\times \set{v} : A\in \sS}$. Then each $\sS_v$ is a strict bramble in $G\cprod H$ and $\ord(\sS_v)=\ord(\sS)=\sbn(G)$. For each node $t\in V(T)$, let
\[\mathdefine{\ell'(t)} \coloneqq \set{v\in V(H) : \ell(t) \text{ is a hitting set for }\sS_v}.\]

We show that $(T,\ell')$ is a lenient tree-decomposition of $H$. Consider any vertex $v\in V(H)$. Since $\sS_v$ is a strict bramble in $G\cprod H$ and $(T,\ell)$ is a lenient tree-decomposition of $G\cprod H$, \cref{fact:lenient-bag-strict-bramble} implies that there exists a node $x\in V(T)$ such that $\ell(x)$ is a hitting set for $\sS_v$, which implies $v\in \ell'(x)$. Suppose there exists a node $y\in V(T-x)$ such that $v\in \ell'(y)$. Then for all $A\times \set{v}\in \sS_v$, since $A\times \set{v}$ is connected and intersects both $\ell(x)$ and $\ell(y)$, \cref{fact:trace-of-connected-sets-ltd} implies that $A\times \set{v}$ intersects $\ell(t)$ for every node $t \in V(xTy)$. Consequently, $\ell(t)$ is a hitting set for $\sS_v$ for all $t\in V(xTy)$. Thus $\set{t\in V(T) : v\in \ell'(t)}$ is non-empty and connected in $T$, and $(T,\ell')$ has the vertex-property. 
Next, assume for a contradiction that $(T,\ell')$ does not have the lenient-property. Then there exists an edge $uv\in E(H)$ such that the sets of nodes $U\coloneqq\set{t\in V(T) : u\in \ell'(t)}$ and $V\coloneqq\set{t\in V(T) : v\in \ell'(t)}$ do not touch. As argued above, $U$ and $V$ are non-empty and connected in $T$. Then there exists $t\in V(T) \setminus (U\cup V)$ that separates $U$ and $V$ in $T$. Since $t\not\in U\cup V$, $\ell(t)$ is neither a hitting set for $\sS_u$ nor for $\sS_v$. Hence there exists $X\times \set{u}\in \sS_u$ and $Y\times \set{v} \in \sS_v$ that are disjoint from $\ell(t)$. Since $U$ and $V$ are non-empty, $(T,\ell)\angle{X\times \set{u}} \cap U\not=\emptyset$ and $(T,\ell)\angle{Y\times \set{v}} \cap V\not=\emptyset$, therefore $(T,\ell)\angle{X\times \set{u}}$ and $(T,\ell)\angle{Y\times \set{v}}$ lie in different components of $T-t$. Then the lenient-property of $(T,\ell)$ implies that $X\times \set{u}$ and $Y\times \set{v}$ do not touch in $G\cprod H$. However since $\sS$ is strict, $X\cap Y\not=\emptyset$, thus $uv\in E(H)$ implies that $X\times \set{u}$ and $Y\times \set{v}$ touch in $G\cprod H$, a contradiction. Therefore $(T,\ell')$ has the lenient-property, and $(T,\ell')$ is a lenient tree-decomposition of $H$ as claimed. Fix $t\in V(T)$ such that $|\ell'(t)|$ is maximum. So $|\ell'(t)| \geq \sbn(H)$ by \cref{thm:sbn-duality}. Then
\[\sbn(G\cprod H) \geq |\ell(t)| \geq \sum_{v\in \ell'(t)}|\ell(t)\cap (V(G)\times \set{v})| \geq \sum_{v\in \ell'(t)}\ord(\sS_v) \geq \sbn(G)\sbn(H). \qedhere\]
\end{proof}

We end with two open problems. The first is about the pathwidth of a Cartesian product:

\begin{question}
\label{ques:cartesian-product-pw}
Does \cref{cor:pathwidth-cartesian-product} hold when $\tw(G)$ is replaced by $\pw(G)$?
\end{question}

Since \citet{HARVEY2018157} showed that $\pw(K_m \cprod K_n) = (\frac{1}{2}+o(1))mn$ for $m\geq n$, a positive answer to \cref{ques:cartesian-product-pw} would be optimal. The second is the question of \citet*{HickingbothamStructuralPropertiesofGraphProducts} that inspired \cref{thm:main-strong-vs-lexicographic}:

\begin{question}[Hickingbotham and Wood \cite{HickingbothamStructuralPropertiesofGraphProducts}]
\label{ques:strong-vs-cart}
Is $\tw(G\boxtimes H) \in \mathcal{O}(\tw(G\cprod H))$?
\end{question}

The following argument due to Neel Kaul (private communication, 2026) shows that $\tw(G\boxtimes H) \leq (\degen(G)+1)(\tw(G\cprod H)+1)-1$, where $\degen(G)$ is the degeneracy of $G$: Let $(T,\beta)$ be a minimum-width tree-decomposition of $G\cprod H$ and consider an orientation of the edges of $G$ such that $\Delta^+(G)=\degen(G)$. For each node $t\in V(T)$, let $\beta'(t)\coloneqq\bigcup_{(g,h)\in \beta(t)} (\set{g}\cup N_G^+(g))\times \set{h}$. Then $(T,\beta')$ is a tree-decomposition $G\cprod H$. Now consider any edge $(g_1, h_1)(g_2, h_2)\in E(G\boxtimes H)\setminus E(G\cprod H)$. Then $g_1g_2\in E(G)$ and $h_1h_2\in E(H)$. Without loss of generality $g_2\in N_G^+(g_1)$. Since $(g_1, h_1)(g_1, h_2)\in E(G\cprod H)$, there exists $t\in V(T)$ such that $\set{(g_1, h_1), (g_1, h_2)}\subseteq \beta(t)$. Hence $\set{(g_1, h_1), (g_2, h_2)} \subseteq \beta'(t)$. Therefore $(T,\beta')$ is a tree-decomposition of $G\boxtimes H$. Finally, since $|\beta'(t)| \leq (\Delta^+(G)+1)|\beta(t)|$ for all $t\in V(T)$, we have $\tw(G\boxtimes H) \leq (\degen(G)+1)(\tw(G\cprod H)+1)-1$. A consequence of this argument is that $\tw(G\cprod H)$ and $\tw(G\boxtimes H)$ are polynomially tied.

\section*{\large Acknowledgements}
Thanks to David Wood and Neel Kaul for fruitful discussions and for their feedback on the paper. Thanks to Jung Hon Yip for feedback on the paper.

\setlength{\bibsep}{0.1ex}

% \bibliography{references}

\end{document}